\documentclass[11pt, a4paper]{article}
\usepackage{}

\usepackage{indentfirst}
\usepackage{lineno}
\usepackage{graphicx}
\usepackage{ae}
\usepackage{amsmath}
\usepackage{amssymb}
\usepackage{latexsym}
\usepackage{url}
\usepackage{epsfig}
\usepackage{cite}
\usepackage{mathrsfs}
\usepackage{amsfonts}
\usepackage{amsthm}
\usepackage{float}
\usepackage{dsfont}

\usepackage{color}

\long\def\delete#1{}

\newcommand{\be}{\begin{equation}}
\newcommand{\ee}{\end{equation}}
\newcommand{\ben}{\begin{equation*}}
\newcommand{\een}{\end{equation*}}
\newcommand{\bea}{\begin{eqnarray}}
\newcommand{\eea}{\end{eqnarray}}
\newcommand{\bean}{\begin{eqnarray*}}
\newcommand{\eean}{\end{eqnarray*}}

\newtheorem{thm}{Theorem}[section]

\newtheorem{lem}[thm]{Lemma}

\newtheorem{defn}[thm]{Definition}

\numberwithin{equation}{section}

\title{Extremality of graph entropy based on Laplacian degrees of $k$-uniform hypergraphs}

\author{Pengli Lu, Yulong Xue
\\
\footnotesize{School of Computer and Communication}\vspace{-0.15cm}\\
\footnotesize{Lanzhou University of Technology}\vspace{-0.15cm}\\
\footnotesize{Lanzhou, 730050, Gansu, P.R. China}\\
\small\texttt{lupengli88@163.com, xue$_-$yu$_-$long@163.com}}
\date{}
\begin{document}

\openup 0.5\jot
\maketitle
%\linenumbers
\begin{abstract}
The graph entropy describes the structural information of graph. Motivated by the definition of graph entropy in general graphs, the graph entropy of hypergraphs based on Laplacian degree are defined. Some results on graph entropy of simple graphs are extended to $k$-uniform hypergraphs. Using an edge-moving operation, the maximum and minimum graph entropy based on Laplacian degrees are determined in $k$-uniform hypertrees, unicyclic $k$-uniform hypergraphs, bicyclic $k$-uniform hypergraphs and $k$-uniform chemical hypertrees, respectively, and the corresponding extremal graphs are determined.

\bigskip


\bigskip

\end{abstract}
%\linenumbers

%====================================================================================================

\section{Introduction}
A hypergraph $\mathcal{H} = (V(\mathcal{H}),E(\mathcal{H}))$ with $n$ vertices and $m$ edges
consists of a set of vertices $V(\mathcal{H})=\{1,2,\cdots,n\}$ and a set of edges $E(\mathcal{H})=\{e_1,e_2,\cdots,e_m\}$. A hypergraph in which
all edges have the same cardinality $k$ is called $k$-uniform. In short, the number of vertices in each edge is $k$. Clearly, ordinary graphs are referred to as 2-uniform hypergraphs. A $k$-uniform hypergraph $\mathcal{H}$ is called simple if there are no multiple edges in $\mathcal{H}$, that is, all edges in $\mathcal{H}$ are distinct. A walk $W$ of length $l$ in $\mathcal{H}$ is a sequence of alternating vertices and edges: $v_0e_1v_1e_2\cdots e_lv_l$,
where $\{v_{i-1}, v_i\}\subseteq e_i$
for $i=1,\cdots, l$. If $v_0=v_l$, then $W$ is called a circuit. A walk of $\mathcal{H}$
is called a path if no vertices and edges are repeated. A circuit $\mathcal{H}$ is called a cycle if no vertices and edges are repeated except $v_0=v_l$. The hypergraph $\mathcal{H}$ is said to be connected if every two vertices are connected by a walk. A walk of $\mathcal{H}$
is called a path if no vertices and edges are repeated. In this paper, we only consider simple and connected $k$-uniform hypergraphs with $k\geq3$.

In a hypergraph, a vertex $v$ is said to be incident to an edge $e$ if $v\in e$. Two vertices
are said to be adjacent if there is an edge that contains both of these vertices. Two edges
are said to be adjacent if their intersection is not empty. For a $k$-uniform hypergraph $\mathcal{H}$, 
the degree $d$ of a vertex $v\in V(\mathcal{H})$ is defined as $d_v=|{e_j:v\in e_j\in E(\mathcal{H})}|$, which means the number of edges containing vertex $v$. A vertex of
degree $d=1$ is called a pendent vertex. Otherwise, it is called a non-pendent vertex. An edge
$e\in E(\mathcal{H})$ is called a pendent edge if $e$ contains exactly $k-1$ pendent vertices. Otherwise, it
is called a non-pendent edge.

Similar to the definition of chemical tree in \cite{Cao}, we define a chemical hypertree is a hypertree with maximum degree $d\leq4$.

\begin{defn}\cite{Hu,Li}
The $k$-th power of an ordinary tree is called a $k$-uniform hypertree.
\end{defn}

\begin{defn}\cite{Hu}
The $k$-th power of star $S_n$, denoted by $S^k_n$, is called a hyperstar.
\end{defn}

\begin{defn}\cite{Fan}
Let $\mathcal{H}$ be a $k$-uniform hypergraph with $n$ vertices, $m$ edges and $l$
connected components. The cyclomatic number of $\mathcal{H}$ is defined to be
$c(\mathcal{H}) = m(k - 1) - n + l$.
So, a hypergraph $\mathcal{H}$ can be called a $c(\mathcal{H})$-cyclic hypergraph.
\end{defn}

Since only simple and connected hypergraphs are considered in this paper, then we have $c(\mathcal{H}) = m(k - 1) - n + 1$.

The Shannon's entropy was first proposed by Shannon to solve the problem of information quantification in communication\cite{Sha}. Then, graph entropy was introduced by Dehmer \cite{Deh}. Based on this, Cao et al. \cite{Cao} gave a novel graph entropy based on
the degrees of graphs as follows.

\begin{defn}\cite{Cao}
Let $G$ be a connected graph with vertex set $V(G) = \{v_1, v_2, \cdots, v_n\}$.
Denote by $d_i$ the degree of the vertex $v_i$. Then the graph entropy based on degrees of $G$ is defined as
$$I_d(G)=-\sum\limits_{i=1}^n(\frac{d_i}{\sum\limits_{i=1}^n d_j}\log_2\frac{d_i}{\sum\limits_{i=1}^n d_j})=\log_2\sum\limits_{i=1}^n d_i-\sum\limits_{i=1}^n(\frac{d_i}{\sum\limits_{i=1}^n d_j}\log_2 d_i)$$
\end{defn}

In \cite{Cao}, they prove some extremal values for the graph entropy of certain families of graphs and find the connection between the graph entropy and the sum of degree powers. In \cite{Das}, they obtain some lower and upper bounds for graph entropy based on degree powers and characterize extremal graphs. In \cite{hu}, they determine the extremal values of some certain families of uniform hypergraph based on degree. There are many researches on the graph entropy, please refer to \cite{X,Wan,das,Che}.

The paper is organized as follows. In Section 2, some concepts and notation in graph theory are introduced. And we define the graph entropy based on the Laplacian degree. In Section 3, the extremal values of graph entropy based on Laplacian degrees of $k$-uniform hypergraphs are determined in hypertrees, unicyclic  hypergraphs, bicyclic hypergraphs and chemical hypertrees, respectively. The paper finishes with a summary and conclusion in Section 4.
%=========================================================================================
\section{Preliminaries}
In this section, we give some definitions and basic results which will be used later.

\begin{defn}\cite{Rod}
Let $A(\mathcal{H})$ denote the adjacency matrix of hypergraph $\mathcal{H}$. Given two distinct vertices $v_i, v_j\in V(\mathcal{H})$
the entry $a_{ij}$ of $A(\mathcal{H})$ is the number of edges in $\mathcal{H}$ containing both $v_i$ and $v_j$; the diagonal entries
of $A(\mathcal{H})$ are zero. Then, define the Laplacian degree of a vertex $v_i\in V(\mathcal{H})$ as $\delta_\ell(v_i)=\sum\limits^n_{j=1} a_{ij}$ .
\end{defn}

Define the non-increasing Laplacian degree sequence of $\mathcal{H}$ by $\pi(\mathcal{H}) = (\delta_1, \delta_2,\cdots,\delta_n)$, that is, $\delta_1\geq \delta_2\geq\cdots \delta_n$. Note that
$$
\sum\limits^n_{i=1}\delta_i=k(k-1)m.
$$

By Definition 1.4, we have graph entropy based on the Laplacian degrees
as follows
\begin{align*}
I_\delta(\mathcal{H})&=-\sum\limits_{i=1}^n(\frac{\delta_i}{\sum\limits_{i=1}^n \delta_j}\log_2\frac{\delta_i}{\sum\limits_{i=1}^n \delta_j})\\&
=\log_2\sum\limits_{i=1}^n \delta_i-\sum\limits_{i=1}^n(\frac{\delta_i}{\sum\limits_{i=1}^n \delta_j}\log_2 \delta_i)\\&
=\log_2 k(k-1)m-\sum\limits_{i=1}^n(\frac{\delta_i}{k(k-1)m}\log_2 \delta_i).
\end{align*}
Thus, for a class of $k$-uniform hypergraphs with given number of edges, in order to determine
the extremal values of graph entropy based on the Laplacian degrees, we just need to determine the extremal values of $\sum\limits_{i=1}^n(\delta_i\log_2 \delta_i)$.
Then we define
$$
h(\mathcal{H})=\sum\limits_{i=1}^n(\delta_i\log_2 \delta_i)
$$
In the following, we first study some properties of $h(\mathcal{H})$.
\begin{defn}\cite{Li}
Let $r\geq1$ and $\mathcal{H}=(V(\mathcal{H}),E(\mathcal{H}))$ be a hypergraph with $u\in V(\mathcal{H})$ and $e_1,\cdots, e_r\in E(\mathcal{H})$ such that $u\notin e_i$ for $i = 1,\cdots, r$. Suppose that $v_i\in e_i$ and write $e'_i = (e_i \setminus \{v_i\} \cup \{u\}), (i = 1,\cdots, r)$. Let $\mathcal{H}'= (V(\mathcal{H}),E'(\mathcal{H}))$ be the hypergraph with $E'(\mathcal{H})=(E(\mathcal{H})\setminus \{e_1,\cdots, e_r\})\cup \{e'_1,\cdots, e'_r\}$. Then we say that $\mathcal{H}'$ is obtained from $\mathcal{H}$ by moving edges $(e_1,\cdots, e_r)$ from $(v_1,\cdots, v_r)$ to $u$.
\end{defn}

\begin{lem}
Let $\mathcal{H}$ be a hypergraph. If there exist two vertices $v_i$ and $v_j$ such that $\delta_i\geq \delta_j+2(k-1)$, where $v_i\in e\in E(\mathcal{H})$
and $v_j\notin e$, then define a new hypergraph $\mathcal{H}'$, which is obtained from $\mathcal{H}$ by the operation of moving edge $e$ from $v_i$ to $v_j$.
Then $h(\mathcal{H})>h(\mathcal{H}')$
\end{lem}

\begin{proof}
Note that $\delta_i\geq \delta_j+2(k-1)$, that is, $\delta_i > \delta_i-(k-1) \geq \delta_j + (k-1) > \delta_j$ . Then
\begin{align*}
h(\mathcal{H})-h(\mathcal{H}')=&\delta_i\log_2\delta_i+\delta_j\log_2\delta_j-(\delta_i-(k-1))\log_2(\delta_i-(k-1))\\&-(\delta_j+(k-1))\log_2(\delta_j+(k-1))\\
=&\{\delta_i\log_2\delta_i-(\delta_i-(k-1))\log_2(\delta_i-(k-1))\}+\{\delta_j\log_2\delta_j\\&-(\delta_j+(k-1))\log_2(\delta_j+(k-1))\}\\
=&(k-1)(\log_2\xi_1+\frac{1}{\ln2})-(k-1)(\log_2\xi_2+\frac{1}{\ln2})\\
=&(k-1)(\log_2\xi_1-\log_2\xi_2)\\
>&0
\end{align*}
where $\xi_1\in(\delta_i-(k-1), \delta_i)$ and $\xi_2\in (\delta_j , \delta_j + (k-1))$. This completes the proof.
\end{proof}

\begin{lem}
Let $\mathcal{H}$ be a hypergraph. If there exist two vertices $v_i$ and $v_j$ such that $\delta_i\geq\delta_j\geq2(k-1)$, where $v_j\in e\in E(\mathcal{H})$
and $v_i\notin e$, then define a new hypergraph $\mathcal{H}'$, which is obtained from $\mathcal{H}$ by the operation of moving edge $e$ from $v_j$ to $v_i$.
Then $h(\mathcal{H})<h(\mathcal{H}')$
\end{lem}

\begin{proof}
Note that $\delta_i\geq\delta_j\geq2(k-1)$, that is, $\delta_i+(k-1) > \delta_i\geq \delta_j > \delta_j-(k-1)$ . Then
\begin{align*}
h(\mathcal{H})-h(\mathcal{H}')=&\delta_i\log_2\delta_i+\delta_j\log_2\delta_j-(\delta_i+(k-1))\log_2(\delta_i+(k-1))\\&-(\delta_j-(k-1))\log_2(\delta_j-(k-1))\\
=&-\{(\delta_i+(k-1))\log_2(\delta_i+(k-1))-\delta_i\log_2\delta_i\}+\{\delta_j\log_2\delta_j\\&-(\delta_j-(k-1))\log_2(\delta_j-(k-1))\}\\
=&-(k-1)(\log_2\xi_1+\frac{1}{\ln2})+(k-1)(\log_2\xi_2+\frac{1}{\ln2})\\
=&-(k-1)(\log_2\xi_1-\log_2\xi_2)\\
<&0
\end{align*}
where $\xi_1\in(\delta_i, \delta_i+k-1)$ and $\xi_2\in (\delta_j-(k-1) , \delta_j)$. This completes the proof.
\end{proof}

\begin{lem}\cite{Li}
Let $\mathcal{H}'$ be a hypergraph obtained from a unicyclic $k$-uniform hypergraph $\mathcal{H}$ by the operation of moving edge. If $\mathcal{H}'$ is connected, then $\mathcal{H}'$ is also a unicyclic $k$-uniform hypergraph.
\end{lem}

\begin{lem}\cite{Li}
Let $\mathcal{H}'$ be a hypergraph obtained from a bicyclic $k$-uniform hypergraph $\mathcal{H}$ by the operation of moving edge. If $\mathcal{H}'$ is connected, then $\mathcal{H}'$ is also a bicyclic $k$-uniform hypergraph.
\end{lem}

Based on lemma 2.5 and lemma 2.6 we know a unicyclic $k$-uniform hypergraph is still a unicyclic $k$-uniform hypergraph after the edge shifting operation in the definition 2.2. After the edge shifting operation in the definition 2.2, the bicyclic $k$-uniform hypergraph is still a bicyclic $k$-uniform hypergraph.

Let $\gamma=\{c_i\}^n_{i=1}$ and $\beta=\{b_i\}^n_{i=1}$ be two sequences of nonnegative integers with $\sum^n_{i=1}c_i=\sum^n_{i=1}b_i$. Assuming the numbers are labeled such that $c_1\geq c_2\geq\cdots\geq c_n$ and $b_1\geq b_2\geq\cdots\geq b_n$, we say that $\gamma$ majorizes $\beta$ $(\gamma\succ\beta)$ if for all $k$
$$
\sum\limits^k_{i=1}c_i\geq\sum\limits^k_{i=1}b_i
$$

Let $\gamma=\{c_i\}^n_{i=1}$ and $\beta=\{b_i\}^n_{i=1}$. If $\gamma$ majorizes $\beta$, then $\sum\limits^k_{i=1}c_i\geq\sum\limits^k_{i=1}b_i$, and the inequality is strict if the majorization is strict.
\begin{lem}\cite{Mar}
Let a real convex function $f$ defined on an interval in $\mathds{R}$, and $x,y\in\mathds{R}^n$. Then
$$
x\prec y\Leftrightarrow\sum\limits_{i=1}^nf(x_i)\leq\sum\limits_{i=1}^nf(y_i)
$$
\end{lem}
%=================================================================================================================================
\section{Extremality of graph entropy based on the Laplacian degrees}\label{SEC:spec}
In this section, we investigate the extremality of graph entropy based on the Laplacian degrees of hypergraphs.\\

3.1 Hypertrees

\begin{thm}
Let $\mathcal{T}$ be a $k$-uniform $(k\geq3)$ hypertree on $n$ vertices with $m=\frac{n-1}{k-1}\geq2$ edges. Then
\begin{align*}
\log_2 [k(k-1)m]-\frac{1}{k(k-1)m}[(mk-m)\log_2(mk-m)+(n-1)(k-1)\log_2(k-1)]\leq I_\delta(\mathcal{T})\leq\\
\log_2[k(k-1)m]-\frac{1}{k(k-1)m}[(m-1)(2k-2)\log_2(2k-2)+(n-m+1)(k-1)\log_2(k-1)]
\end{align*}
where the first equality holds if and only if $\mathcal{T}\cong S_{m+1}^k$, and the second equality holds if and only if $\mathcal{T}\cong P_n$.
\end{thm}

\begin{proof}
Let $\mathcal{T}$ be a tree of order $n$ and let its Laplacian degree sequence be $\delta_1, \delta_2,\cdots,\delta_n$. It is not difficult to see that if $\mathcal{T}\ncong P_n$, then there must exist a pair $\delta_i$ and $\delta_j$ such that $\delta_i\geq \delta_j+2(k-1)$. We construct a tree $\mathcal{T}'$ by replacing the pair $\delta_i$ and $\delta_j$ by the pair $\delta_i-(k-1)$ and $\delta_j+(k-1)$.
Then by Lemma 2.3, we obtain that $h(\mathcal{T})> h(\mathcal{T}')$. Repeating the above operation until there is no pair $\delta_i$ and $\delta_j$ such that
$\delta_i-(k-1)$ and $\delta_j+(k-1)$ for all $i, j$, we obtain a tree sequence $T, T', T'_1,\cdots, T'_s$ such that $T'_s \cong P_n$. Clearly, $h(T)>h(T')> h(T'_1)>\cdots>h(T'_s)$. Thus, for any tree $\mathcal{T}\ncong P_n$, $h(\mathcal{T})> h(P_n)$. Note that Laplacian degree sequence $\pi(P_n)=[2k-2^{(m-1)}, k-1^{(n-m+1)}]$.

Using Lemma 2.4 and by a similar discussion, we can show that if $\mathcal{T}$ is a tree of order $n$, then any tree $\mathcal{T}\ncong S_{m+1}^k$, $h(\mathcal{T})< h(S_{m+1}^k)$. Note that Laplacian degree sequence $\pi(S_{m+1}^k)=[mk-m, k-1^{(n-1)}]$.

From the discussion above, we note that only $\mathcal{T}\cong S_{m+1}^k$ can attain the maximum
value of graph entropy among all $k$-uniform $(k\geq3)$ hypertrees on $n$ vertices with $m=\frac{n-1}{k-1}\geq 2$ edges. And only $\mathcal{T}\cong P_n$ can attain the minimum value of graph entropy among all $k$-uniform $(k\geq3)$ hypertrees on $n$ vertices with $m=\frac{n-1}{k-1}\geq 2$ edges. So we have
\begin{align*}
\log_2 [k(k-1)m]-\frac{1}{k(k-1)m}[(mk-m)\log_2(mk-m)+(n-1)(k-1)\log_2(k-1)]\leq I_\delta(\mathcal{T})\\
\leq\log_2[k(k-1)m]-\frac{1}{k(k-1)m}[(m-1)(2k-2)\log_2(2k-2)+(n-m+1)(k-1)\log_2(k-1)]
\end{align*}
This completes the proof.\\
\end{proof}
%=========================================================================================================================
3.2 Unicyclic hypergraphs

Let $\mathcal{H}^{\divideontimes}$ be a unicyclic $k$-uniform hypergraph with Laplacian degree sequence $\pi(\mathcal{H}^{\divideontimes})=[2k-2^{(m)}, k-1^{(n-m)}]$. Denote by $\mathcal{H}^{\divideontimes\divideontimes}$ the unicyclic $k$-uniform hypergraph with Laplacian degree sequence $\pi(\mathcal{H}^{\divideontimes\divideontimes})=[2k-3^{(2)}, k-1^{(n-2)}]$.
\begin{thm}
Let $\mathcal{H}$ be a unicyclic $k$-uniform $(k\geq3)$ hypergraph on $n$ vertices with $m=\frac{n}{k-1}\geq2$ edges. Then
\begin{align*}
\log_2 [k(k-1)m]-\frac{1}{k(k-1)m}\{2(2k-3)\log_2(2k-3)+(n-2)(k-1)\log_2(k-1)\}\leq I_\delta(\mathcal{H})\\
\leq\log_2[k(k-1)m]-\frac{1}{k(k-1)m}\{2m(k-1)\log_2[2(k-1)]+(n-m)(k-1)\log_2(k-1)\}
\end{align*}
where the first equality holds if and only if $\mathcal{H}\cong\mathcal{H}^{\divideontimes\divideontimes}$, and the second equality holds if and only if $\mathcal{H}\cong\mathcal{H}^{\divideontimes}$.
\end{thm}

\begin{proof}
Let $f(x)=x\log_2 x$ for $x\geq0$. Obviously, $f(x)$ is strictly convex. So $h(\mathcal{H})$ satisfy the conclusion of Lemma 2.7.
The larger the first $k$ term of the non-increasing Laplacian degree sequence $\pi(\mathcal{H}) = (\delta_1, \delta_2,\cdots,\delta_n)$ is, then the bigger $h(\mathcal{H})$ is. On the contrary, the result is also the opposite. In other words, because the sum of degrees is $k(k-1)m$, then there is the sequence with more $k-1$, $h(\mathcal{H})$ is the larger.

And we know, there are three kinds of structures in the unicyclic $k$-uniform hypergraphs: multi-pendent edges, single pendent edge and no pendent edge. It is not difficult to see that the non-increasing Laplacian degree sequence without pendent edge have the least numbers of term $k-1$, which equals $n-m$. So when unicyclic $k$-uniform hypergraph have no pendent edge, the Laplacian degree sequence $\pi(\mathcal{H}^{\divideontimes})=[2k-2^{(m)}, k-1^{(n-m)}]$, $h(\mathcal{H})$ has the minimum value.
\begin{align*}
h(\mathcal{H})\geq 2m(k-1)\log_2[2(k-1)]+(n-m)(k-1)\log_2(k-1)
\end{align*}

When the number of edges contained on the cycle is the same, the non-increasing Laplacian degree sequence of multiple pendent edges has more $k-1$ terms than that of single pendent edge. We consider the extreme state in the case of multiple pendent edges. In a unicyclic $k$-uniform hypergraph with multiple pendent edges, the number of $k-1$ is up to $n-2$, the Laplacian degree sequence $\pi(\mathcal{H}^{\divideontimes\divideontimes})=[2k-3^{(2)}, k-1^{(n-2)}]$. Therefore, in this case, $h(\mathcal{H})$ gets the maximum value.
\begin{align*}
h(\mathcal{H})\leq 2(2k-3)\log_2(2k-3)+(n-2)(k-1)\log_2(k-1)
\end{align*}

From the discussion above, we have
\begin{align*}
\log_2 [k(k-1)m]-\frac{1}{k(k-1)m}\{2(2k-3)\log_2(2k-3)+(n-2)(k-1)\log_2(k-1)\}\leq I_\delta(\mathcal{H})\\
\leq\log_2[k(k-1)m]-\frac{1}{k(k-1)m}\{2m(k-1)\log_2[2(k-1)]+(n-m)(k-1)\log_2(k-1)\}
\end{align*}
This completes the proof.\\
\end{proof}
%==============================================================================
3.3 Bicyclic hypergraphs

Let $\mathcal{H}^*$ be a bicyclic $k$-uniform hypergraph with Laplacian degree sequence $\pi(\mathcal{H}^*)=[2k-2^{(m+1)}, k-1^{(n-m-1)}]$. Denote by $\mathcal{H}^{**}$ the bicyclic $k$-uniform hypergraph with Laplacian degree sequence $\pi(\mathcal{H}^{**})=[mk-m, 2k-3^{(2)}, k-1^{(n-3)}]$.

\begin{lem}
Let $\mathcal{H}$ be a bicyclic $k$-uniform hypergraph of order n. Then we have $h(\mathcal{H})\geq h(\mathcal{H}^*)$, the equality holds if and only if $\mathcal{H}\cong\mathcal{H}^*$; and $h(\mathcal{H})\leq h(\mathcal{H}^{**})$, the equality holds if and only if $\mathcal{H}\cong\mathcal{H}^{**}$.
\end{lem}

\begin{proof}
We do the proof by contradiction. Suppose $\mathcal{H}$ attains the minimum value among all bicyclic $k$-uniform hypergraphs. We claim that there exists at least one vertex $u$ fulfills $\delta_u\geq3k-3$. Otherwise, let $e$ be a non-pendent edge containing $u$ and  $C=v_1e_1v_2e_2\cdots e_tv_{t+1}$ $(v_{t+1}=v_1)$ be a cycle in $\mathcal{H}$.

\textbf{Case 1.} $u\in C$.

Without loss of generality, we assume $e = e_1$ and $u\in e_1$. Find a longest path $P =
uf_1u_1f_2\cdots f_su_s$ starting at $u$ such that $u_i \notin C$ and $f_i \notin C$ for $i = 1, 2, \cdots , s$. Obviously, $\delta_{u_s} = k-1$.

\textbf{Subcase 1.1} $u = v_1$ or $v_2$.

Without loss of generality, we assume $u = v_1$. Then $\delta_{v_1}\geq 3$. Let $\mathcal{H}' = (V(\mathcal{H}),E(\mathcal{H}'))$ be the hypergraph with $E(\mathcal{H}') = (E(\mathcal{H})\setminus \{e_1\})\cup \{e'_1\}$, where $e'_1 = (e_1\setminus \{v_1\})\cup \{u_s\}$. Then, by Lemma 2.3, we obtain that $h(\mathcal{H}) > h(\mathcal{H}')$, a contradiction.

\textbf{Subcase 1.2} $u\neq v_1, v_2$.

Let $f_1\neq e_1$ denote another edge containing $u$. Let $\mathcal{H}' = (V(\mathcal{H}),E(\mathcal{H}') )$ be the hypergraph with $E(\mathcal{H}') = (E(\mathcal{H})\setminus \{e_1,f_1\})\cup \{e'_1,f'_1\}$, where $e'_1 = (e_1\setminus \{v_1\})\cup \{u_s\}$ and $f'_1 = (f_1\setminus \{u\})\cup \{v_1\}$.
Then, by Lemma 2.3, we obtain that $h(\mathcal{H}) > h(\mathcal{H}')$, a contradiction.

\textbf{Case 2.} $u \notin C$.

Find a longest path $P = vf_1u_1f_2\cdots f_su_s$ ($u_i\notin C$ and $f_i\notin C$ for $i = 1, 2, \cdots , s$) starting
at $v$ such that $v\in e_i$ for some $i\in \{1, 2,\cdots , t\}$ and $f_{j-1}\cap f_j = \{u\}$ for some $j\in\{1, 2, \cdots , s\}$.
Let$\mathcal{H}' = (V(\mathcal{H}),E(\mathcal{H}') )$ be the hypergraph with $E(\mathcal{H}') = (E(\mathcal{H})\setminus \{e_1,f_j\})\cup \{e'_1,f'_j\}$, where $e'_1 = (e_1\setminus \{v_i\})\cup \{u_s\}$ and $f'_j = (f_j\setminus \{u\})\cup \{v_i\}$. Then, by Lemma 2.3, we obtain that $h(\mathcal{H}) > h(\mathcal{H}')$, a contradiction.

From the discussion above, we note that only $\mathcal{H}\cong\mathcal{H}^*$, can attain the minimum value of $h(\mathcal{H})$. By simple computation, we have $h(\mathcal{H})\geq(m+1)(2k-2)\log_2(2k-2)+(n-m-1)(k-1)\log_2(k-1)$.

Suppose $\mathcal{H}$ attains the maximum value among all bicyclic $k$-uniform hypergraphs. Let $u$ be a vertex of the maximum  Laplacian degree.

\textbf{Claim 1.} If $e\in E(\mathcal{H})$ has $k-1$ pendent vertices, then $u\in e$.

On the contrary, we suppose that $e\in E(\mathcal{H})$ has $k-1$ pendent vertices but $u\notin e$. Let $w\in e$ be the
vertex with $\delta_w > k-1$. Construct a new hypergraph $\mathcal{H}' = (V(\mathcal{H}),E(\mathcal{H}'))$ with $E(\mathcal{H}') = (E(\mathcal{H})\setminus \{e\})\cup \{e'\}$, where $e' = (e\setminus \{w\})\cup \{u\}$. By Lemma 2.4, we have $h(\mathcal{H}) < h(\mathcal{H}')$, a contradiction.

\textbf{Claim 2.} $u$ is a common vertex of all the cycles of $\mathcal{H}$.

On the contrary, assume $u$ does not lie on the cycle $C = v_1e_1v_2e_2\cdots e_tv_{t+1} (v_{t+1} = v_1)$ of $\mathcal{H}$. Find a path $P = vf_1u_1f_2\cdots u_{s-1}f_su$ ($u_i\notin C$ for $i = 1, 2, \cdots , s-1$ and $f_j\notin C$ for $j = 1, 2, \cdots , s$) starting at $v$ such that $v\in e_i$ for some $i\in \{1, 2,\cdots , t\}$. Let $\mathcal{H}' = (V(\mathcal{H}),E(\mathcal{H}'))$ be the hypergraph with $E(\mathcal{H}') = (E(\mathcal{H})\setminus \{e_i,f_1\})\cup \{e'_i,f'_1\}$ where $e'_i = (e_i\setminus \{v_i\})\cup \{u\}$ and $f'_1 = (f_1\setminus \{v\})\cup \{v_i\}$. By Lemma 2.4, we have $h(\mathcal{H}) < h(\mathcal{H}')$, a contradiction.

\textbf{Claim 3.} $|e(C)| = 2$ for any cycle $C$ in $\mathcal{H}$, where $e(C) = \{e \mid e$ is an edge $\in C\}$.

We use $C = v_1e_1v_2e_2\cdots e_tv_{t+1}$ $(v_{t+1} = v_1)$ to denote a cycle in $\mathcal{H}$. Without
loss of generality, suppose that $v_2\neq u\in e_1$ ($u$ can be equal to $v_1$). If $|e(C)|\geq 3$,
then let $\mathcal{H}' = (V(\mathcal{H}),E(\mathcal{H}'))$ be the hypergraph with $E(\mathcal{H}') = (E(\mathcal{H})\setminus \{e_2\})\cup \{e'_2\}$, where $e'_2 = (e_2\setminus \{v_2\})\cup \{u\}$. By Lemma 2.4, we have $h(\mathcal{H}) < h(\mathcal{H}')$, a contradiction.

From the discussion above, we note that only $\mathcal{H}\cong\mathcal{H}^{**}$, can attain the maximum value of $h(\mathcal{H})$. By simple computation, we have $h(\mathcal{H})\leq (mk-k)\log_2(mk-k)+(4k-6)\log_2(2k-3)+(n-3)(k-1)\log_2(k-1)$. This completes the proof.

\end{proof}

By the definition of the graph entropy based on the Laplacian degrees and Lemma 3.3, we have the following result.

\begin{thm}
Let $\mathcal{H}$ be a bicyclic $k$-uniform $(k\geq3)$ hypergraph on $n$ vertices with $m=\frac{n+1}{k-1}\geq2$ edges. Then
\begin{align*}
\log_2 [k(k-1)m]-\frac{1}{k(k-1)m}\{(mk-k)\log_2(mk-k)+(4k-6)\log_2(2k-3)+(n-3)(k-1)\\
\log_2(k-1)\}\leq I_\delta(\mathcal{H})
\leq\log_2[k(k-1)m]-\frac{1}{k(k-1)m}\{(m+1)(2k-2)\log_2(2k-2)+\\(n-m-1)(k-1)\log_2(k-1)\}
\end{align*}
where the first equality holds if and only if $\mathcal{H}\cong \mathcal{H^{**}}$, and the second equality holds if and only if $\mathcal{H}\cong\mathcal{H^*}$.
\end{thm}

3.4 Chemical hypertrees

We define $\mathscr{T}^*$ as follows: $\mathscr{T}^*$ is a hypertree with $n$ vertices and $m-1=3a+i$, $(i=0,1,2)$, whose Laplacian degree sequence  $\pi(\mathscr{T}^*)=[4k-4^{(a)},(i+1)(k-1),k-1^{(n-a-1)}]$.
\begin{lem}
Let $\mathscr{T}$ be a chemical $k$-uniform $(k\geq3)$ hypertree on $n$ vertices. Then we get $h(\mathscr{T})\leq h(\mathscr{T}^*)$, the  equality holds if and only if $\mathscr{T}\in \mathscr{T}^*$.
\end{lem}

\begin{proof}
From the discussion above, we know that $h(\mathscr{T})$ satisfy the conclusion of Lemma 2.7. The more number of vertices with a large degree, the greater $h(\mathscr{T})$ is. Because the maximum degree of chemical hypertree at most 4, then the more the number of  Laplacian degree $4k-4$ is, the greater the $h(\mathscr{T})$ is. If there exists a pair $(d_i,d_j)$ such that $4>d_i\geq d_j\geq2$, then we can construct a new hypertree $\mathscr{T}'$ by  replacing the pair $(d_i,d_j)$ by $(d_i+1,d_j-1)$. By Lemma 2.4, we can also proof $h(\mathscr{T})< h(\mathscr{T}')$. Now we denote by $a,b,c,d$ the number of the vertices of degrees $4,3,2,1$, respectively. Then we have
\begin{align*}
\left\{
\begin{array}{rl}
&4a+3b+2c+d=km\\
&a+b+c+d=n\\
&b+c\leq1.
\end{array}
\right.
\end{align*}
From the above expressions, we infer the following solutions:

Case 1. If $m-1\equiv0$ $(mod\ 3)$, then $a=\frac{m-1}{3}$, $b=c=0$, $d=n-a$.

Case 2. If $m-1\equiv0$ $(mod\ 3)$, then $a=\frac{m-2}{3}$, $b=0$, $c=1$, $d=n-a-1$.

Case 3. If $m-1\equiv0$ $(mod\ 3)$, then $a=\frac{m-3}{3}$, $b=1$, $c=0$, $d=n-a$.

Therefore, the proof is completed.
\end{proof}

\begin{thm}
Let $\mathscr{T}$ be a chemical $k$-uniform $(k\geq3)$ hypertree on $n$ vertices with $m=\frac{n-1}{k-1}$ edges. Then
\begin{align*}
\log_2[k(k-1)m]-\frac{1}{k(k-1)m}[a(4k-4)\log_2(4k-4)+(i+1)(k-1)\log_2(i+1)(k-1)+\\(n-a-1)(k-1)\log_2(k-1)]\leq I_{\delta}(\mathscr{T})\leq\log_2[k(k-1)m]-\frac{1}{k(k-1)m}[(m-1)(2k-2)\\\log_2(2k-2)+(n-m+1)(k-1)\log_2(k-1)],
\end{align*}
where the first equality holds if and only if $\mathscr{T}\in \mathscr{T^{*}}$, and the second equality holds if and only if $\mathscr{T}\cong P_n$.
\end{thm}

\begin{proof}
Observe that the hypergraph $P_n$ is also a chemical hypertree. Then from theorem 3.1, we know that $P_n$ attains the maximum value of $I_{\delta}(\mathscr{T})$ among all chemical trees. Based on lemma 3.5 we know $I_{\delta}(\mathscr{T})$ can obtain the minimum value if $\mathscr{T}\in \mathscr{T^{*}}$. Then
\begin{align*}
\log_2[k(k-1)m]-\frac{1}{k(k-1)m}[a(4k-4)\log_2(4k-4)+(i+1)(k-1)\log_2(i+1)(k-1)+\\(n-a-1)(k-1)\log_2(k-1)]\leq I_{\delta}(\mathscr{T})\leq\log_2[k(k-1)m]-\frac{1}{k(k-1)m}[(m-1)(2k-2)\\\log_2(2k-2)+(n-m+1)(k-1)\log_2(k-1)].
\end{align*}
This completes the proof.

\end{proof}

%=================================================================================

%-----------------------------------------------------------------------------------

\section{Conclusions}
In this paper, we extend some results of simple graphs to $k$-uniform hypergraphs. By using the properties of Laplacian degrees and the operation of moving edges on hypergraphs, we obtain the extremal values of graph entropy based on Laplacian degrees of $k$-uniform hypergraphs in hypertrees, unicyclic  hypergraphs and bicyclic hypergraphs, respectively. In the next work, the above conclusions can be extended to $n$-cyclic $k$-uniform hypergraphs, but because there are many different structures in $n$-cyclic $k$-uniform hypergraphs, the Laplacian degree sequences will be very changeable, which is a difficult point in the study.
%====================================================================================

\end{document}